\documentclass[12pt]{amsart}
\usepackage{amsmath,amsthm,amsfonts,amscd,amssymb,eucal,latexsym}
\usepackage{mathrsfs}

\newtheorem{theorem}{Theorem}

\newtheorem{corollary}[theorem]{Corollary}

\newtheorem{definition}[theorem]{Definition}

\newtheorem{lemma}[theorem]{Lemma}

\newtheorem{remark}[theorem]{Remark}

\def\H{{\mathscr{H} }}

\def\L{{\mathscr{L}}}

\def\M{{\mathscr{ M}}}
\def\N{{\mathscr{N}}}
\def\F{{\mathscr{F}}}
\def\P{\Phi}

\def\A{{\mathscr{A} }}\def\B{{\mathscr{B} }}

\def\M{{\mathscr{M} }}

\def\t{{\tau }}
\def\S{{\mathscr{S} }}
\def\T{{\mathscr{T} }}

\def\R{{\mathscr{R} }}

\def\R{\mathscr{R}}

\begin{document}

\title{On a von Neumann algebra which is a complemented subspace}

\author{Erik Christensen}
\address{Department of Mathematical Sciences, University of Copenhagen, Copenhagen, Denmark}
\email{E-mail: echris@math.ku.dk}

\author{Liguang Wang}
\address{School of Mathematical Sciences, Qufu Normal University, Qufu, 273165, China}
\email{E-mail: wangliguang0510@163.com}
\thanks{Partially supported  by NSFC (11371222) and NSF of Shandong Province (ZR2012AM024).}

 \date{today}
\keywords{Type II$_1$ factor; Fundmental group; Hyperfinite type II$_1$ factor; Injective von Neumann algebra; Complemented subspace.}

\begin{abstract}
Let $\M$ be a von Neumann algebra of type II$_1$ which is also a
complemented subspace of $\B(\H).$ We establish an algebraic
criterion,  which ensures that $\M$ is an injective von Neumann
algebra. As a corollary we
show that if $\M$ is a complemented factor of type II$_1$ on a Hilbert space $\H$, then $\M$ is injective if its fundamental group  is non-trivial.
\end{abstract}

\subjclass[2000]{ 46L10, 46L50}

\maketitle

\baselineskip=15pt \vskip0.1cm
\section{Introduction}

In the early works \cite{MVN36}-\cite{MVN43}  by Murray and von Neumann, they realized that there is a certain sort of {\em rings of operators } which to a large extent behave like the algebras $M_n(\mathbb{C})$ consisting of all complex $n \times n$  matrices, except that the natural dimension function now has the image $[0, 1]$ instead of the set $\{0, 1, \dots , n\}$. Today $\emph{rings of operators}$ are called von Neumann algebras and the ones with a continuous dimension function with values in $[0, 1]$ are called  von Neumann algebras of type II$_1$. Factors are von Neumann algebras whose centers consist of scalar multiples of the identity. Finite-dimensional factors are (isomorphic to) full matrix algebras. Infinite-dimensional factors admitting a positive and bounded trace  are called factors of type II$_1$. Murray and von Neumamnn also realized that there are at least two non-isomorphic factors  of type II$_1$, namely the free group factor $\L(\mathbb{F}_2)$ and the hyperfinite type II$_1$ factor $\R$, where $\L(\mathbb{F}_2)$ is the  von Neumann algebra obtained by taking the ultraweak closure of the left regular representation of the non-abelian free group $\mathbb{F}_2$ on 2 generators,  while the  hyperfinite type II$_1$ factor $\R$ was constructed  as the ultraweak closure of an increasing sequence of full matrix algebras.

The hyperfinite type II$_1$ factor $\R$ is easily proven to be  an
injective object in the category of C$^{\ast}$-algebras and completely
positive mappings (see  Paulsen's book  \cite{Pa}),  and it turns out
\cite{To} that a C$^{\ast}$-algebra $\A$ acting on a Hilbert space
$\H$ is  injective  if and only if there exists a projection $\Pi :
\B(\H) \to \A$ of norm 1. Such a projection turns out to be completely
positive and $\A$-modular, in the sense that $\Pi(AT) = A\Pi(T) $ and
$\Pi(TA) = \Pi(T)A$ for all $T$ in $\B(\H)$ and $A$ in $\A$.  The
fundamental work by Connes \cite{C1976} showed among other things,
that all injective type II$_1$ factors on a separable Hilbert space
are isomorphic to the hyperfinite type II$_1$ factor $\R$.   Based on
parts of this work, Connes gave in \cite{C1978} a cohomological
characterization of the injective type II$_1$ factor. The module
Connes used as coefficients for the cohomology consists of bounded
mappings on $\B(\H)$ which are modular for the commutant $\M^\prime.$
This module was then studied by Bunce and Paschke in \cite{BP} and
they proved  that if a von Neumann algebra $\M$ on a Hilbert space
$\H$ is the image of a bounded $\M$-module projection $\Pi : \B(\H)
\to \M,$ then $\M$ is injective. It is then a natural question to ask
if a von Neumann algebra $\M$ on a Hilbert space $\H$, which is a
complemented subspace of $\B(\H) $,  must be injective. For such an
algebra, there exists a bounded projection $\Pi: \B(\H) \to \M$ and it
was proven - by quite different methods - by Pisier in \cite{P3} and
by the first named author and Sinclair in \cite{CS1}, that if the
projection $\Pi$ is completely bounded (see \cite{Pa}), then $\M$ is
an injective von Neumann algebra. In Corollary 4.6 of \cite{HP}, which was published
before the just named result, Haagerup and Pisier showed that if a
von Neumann algebra $\M$ is isomorphic to $M_n(\M)$ for some natural
number $n \geq 2$ and also is a complemented subspace then there exists
a completely bounded projection from $\B(\H)$ onto $\M.$ It then
follows that if a type II$_{1}$ factor $\M$ is a complemented subspace
and its fundamental group contains an integer $n \geq 2$ then it is
injective.

It is well known that type I von Neumann  algebras are injective and
the results on completely bounded mappings yield  quite easily that if
a properly infinite von Neumann algebra is a complemented subspace,
then it is an injective von Neumann algebra. For a von Neumann algebra
$\M$ which is a complemented subspace of $\B(\H)$, the question of
injectivity is then reduced to the case of von Neumann algebra of type
II$_1.$   Some partial results, which all show injectivity, exist
(\cite{CS2}, \cite{C1}, \cite{HP}, \cite{P1}, \cite{P2} ) and in this article we will add some more, which are based on properties of  certain $\ast$-endomorphisms on $\M.$

We show that if a von Neumann algebra $\M$ of  type II$_1$  on a
Hilbert space $\H$ is a complemented subspace of $\B(\H)$ and $\Phi:
\M \to \M $ is an injective normal $\ast$-endomorphism which has the
property that for the center-valued trace $\mathrm{Tr}$ of $\M$ onto
the center of $\M$,  we have  $\|\mathrm{Tr}(\Phi(I))\|< 1,$
then $\M$ is injective. This result may be applied to the factor case
and we show that if $\M$ is a factor of type II$_1$ on a Hilbert space
$\H$ and there is a bounded projection from $\B(\H)$ onto $\M$, then
$\M$ is injective, if the fundamental group of $\M$ is non-trivial.

The proof of the theorem is based on a trick which shows the existence of a completely bounded projection from $\B(\H)$ onto $\M$. The methods used in the trick are very much inspired by some matrix constructions, which were used by Pisier in his studies on the similarity degree. Examples of such constructions may be found in the article \cite{P4}.

\section{Complemeted von Neumann algebras and endomorphisms}

In this section $\M$ will always denote a von Neumann algebra of type II$_1$ on a Hilbert space $\H$ and $\Phi : \M \to \M$ a  normal and injective  $\ast$-endomorphism. We will define the projection $E$ in $\M$ by $E:= \Phi(I)$ and let $\mathrm{Tr}: \M \to \M \cap \M^{\prime}$ denote the unique center-valued trace of $\M$. For a linear mapping $\Gamma $ between operator algebras, we will use the notation $\Gamma_n := \Gamma \otimes \mathrm{id}_{M_n(\mathbb{C})}$ and $\|\Gamma\|_n := \|\Gamma_n\|.$

It is our aim to prove the following theorem.

\begin{theorem} \label{TTT} Suppose $\Pi: \B(\H) \to \M$ is a bounded projection and  $\Phi: \M\rightarrow \M$ is a normal injective  $\ast$-homomorphism. If $\|\mathrm{Tr}(\Phi(I))\|< 1$,  then $\M$ is  injective.
\end{theorem}

We will base the proof of Theorem \ref{TTT} on some steps which we formulate as individual lemmas.

\begin{lemma} \label{EEE} Let $n$ be a natural number. Without loss of generality we may assume that $\|\mathrm{Tr}(E)\| < \frac{1}{n}$ and that there exists a surjective  isometry $V: \H \to E(\H) $ such that for all $T \in \M $ we have $\Phi(T) = VTV^*.$ \end{lemma}

\begin{proof} Since $\Phi$ is faithful and $E \neq 0$, we have  $0 <  \alpha :=  \|\mathrm{Tr}(E)\| < 1.$ Since $\Phi$ is a normal isomorphism onto its image, say $\N$,  we get that $E$  is the unit of $\N$ and for $F:= \Phi(E)$ we get $ {\mathrm{Tr}}_{\N}(F) \leq \alpha E.$ Hence $ \mathrm{Tr}_{\M}(F) \leq  \alpha^2 I.$ This process may be iterated and we see that there exists a natural number $k$ such that $\| \mathrm{Tr}(\Phi^k(I))\| < \frac{1}{n}$.   A replacement of $\Phi$ by $ \Phi^k$ proves the first claim.

With respect to the second claim, we remark that if we look at the amplifications of $\M$ and $\Phi$ to the algebra $\tilde \M : = \M \otimes \mathbb{C} I$ on $\H \otimes l^2(\mathbb{N})$, then $\tilde \Phi $ is implemented by a surjective isometry $V: \H \otimes l^2(\mathbb{N}) \to E(\H) \otimes l^2(\mathbb{N})$. Since $\widetilde{\B(\H)}$  is an injective von Neumann algebra, it follows that $\tilde \M $ is complemented inside $\B(\H \otimes l^2(\mathbb{N}))$,  and thus we may as well assume that $\Phi(T) = VTV^{\ast}$ for $T\in \M$,  as claimed. \end{proof}

For projections $P$ and $Q$ in $\M$,
we say $P$ is weaker than $Q$ (and write   $P\precsim Q$) when $P$ is (Murray-von Neumann) equivalent to a subprojection of $Q$ (see Definition 6.2.1 of \cite{KR}).   The following lemma is probably well known, but we do not have a reference at hand.

\begin{lemma} \label{BBB} Let $P$ and $Q$ be projections in $\M$. Then $P\precsim Q$ if and only if $\mathrm{Tr}(P)\leq \mathrm{Tr}(Q)$.
\end{lemma}

\begin{proof}   Suppose $P\precsim Q$. Then there exists a partial isometry $W$ in $\M$ such that $W^{\ast}W=P$  and $WW^{\ast}\leq Q$. Since $\mathrm{Tr}$ is positive and a trace $\big(\mathrm{Tr}(S^{\ast}S)=\mathrm{Tr}(SS^{\ast})$ for all $S$ in $\M\big)$, we get $\mathrm{Tr}(P)\leq \mathrm{Tr}(Q)$.

Suppose that $\mathrm{Tr}(P)\leq \mathrm{Tr}(Q)$ and $P$ is not weaker than $ Q$. By the Comparison Theorem (see Theorem 6.2.7 in \cite{KR}), there exists a nonzero central projection $Z$ such that $ ZQ\precsim ZP$ and $ZQ \nsim ZP$. Hence there exists a partial isometry $W$ in $\M$ such that $W^{\ast}W=ZQ$, $WW^{\ast}\leq ZP$  and $WW^{\ast}\neq ZP$. Since $\mathrm{Tr}$ is a faithful center-valued trace and center-modular, we get $\mathrm{Tr}(Z(P-Q))\geq 0$ and $\mathrm{Tr}(Z(P-Q))\neq 0$. On the other hand, $$\mathrm{Tr}(Z(P-Q))=Z\mathrm{Tr}(P-Q)=Z(\mathrm{Tr}(P)-\mathrm{Tr}(Q))\leq 0$$ which is a contradiction. So $P\precsim Q$   and the lemma follows.
\end{proof}

\begin{lemma} \label{CCC} Let $n$ be a natural number. If  $||\mathrm{Tr}(\Phi(I))||=||\mathrm{Tr}(E)||<\frac{1}{n},$
then there is a set $\{E_1, E_2, \cdots, E_n\}$ of pairwise orthogonal and equivalent projections in $\M$ such that $E_1=E$.
\end{lemma}

\begin{proof} Let $\{E_1, E_2, \cdots, E_k\}$ be a maximal family of pairwise orthogonal and equivalent projections in $\M$ with $E_1=E$. Such a family must be finite since $\M$ is a finite von Neumann algebra. Let $$F=I-E_1-E_2-\cdots-E_k.$$ If $k<n$, then \begin{align}\mathrm{Tr}(F)&=I-k\mathrm{Tr}(E)\geq I-(n-1)\mathrm{Tr}(E)\nonumber\\
&\geq I-\frac{n-1}{n}I=\frac{1}{n}I\geq \mathrm{Tr}(E).\nonumber\end{align}
Hence by  Lemma \ref{BBB}, $E\precsim F$ and the family $\{E_1, E_2, \cdots, E_k\}$ is not maximal which  is a contradiction. Therefore $k\geq  n$ and the lemma follows.
\end{proof}

\begin{lemma} \label{AAA} Suppose $\Pi : \B(\H) \to \M$ is a bounded projection and $\A$ is a finite dimensional von Neumann subalgebra of $\M$, which has the same unit as  $\M.$ Then there exists a bounded projection $\Psi : \B(\H) \to \M$ which is $\A$-modular and satisfies $\|\Psi\| \leq \|\Pi\|.$
\end{lemma}

\begin{proof} Let $G=U(\A)$ be the group of all unitary operators in $\A$. Then $G$ is a compact group and there is a Haar probability measure $\mu$ on $G$. For $S\in \B(\H)$, we define
\begin{displaymath}\Psi(S)=\int_{U_1\in G}\int_{U_2\in G}U_1\Pi(U_{1}^{\ast}S U_2)U_{2}^{\ast}d\mu(U_1)d\mu(U_2).\end{displaymath}
Then  for $T \in \M$,  we have $\Psi(T) = T$ and for $U_1, U_2\in G$, $S\in \B(\H)$, $$\Psi(U_1SU_2)=U_1\Psi(S)U_2$$ and the lemma follows. \end{proof}

\begin{lemma} \label{DDD}  If $||\mathrm{Tr}(E)||<\frac{1}{n}$, then there exists a projection $\Gamma$ of $\B(\H)$ onto $\M$ such that $||\Gamma||_n\leq ||\Pi||$.
\end{lemma}

\begin{proof}  It follows from Lemma \ref{CCC} that there exists a set $\{E_1, E_2, \cdots, E_n\}$ of pairwise orthogonal and equivalent projections in $\M$ such that $E=E_1$. We may supplement this set to a set of matrix units $\{E_{ij}: 1\leq i, j\leq n\}$ for which $E_{ii}=E_i$ and define $F=I-E_1-E_2-\cdots -E_n$. Then the algebra $\A$ defined as
$$\A=\mathrm{span}\{F\cup \{E_{ij}: 1\leq i, j\leq n\}\}$$
is a finite dimensional von Neumann subalgebra of $\M$ having the same unit as $\M,$ and by Lemma \ref{AAA} we have a projection $\Psi: \B(\H)\rightarrow \M$ such that $\Psi$ is $\A$-modular and $||\Psi||\leq ||\Pi||$. Let $\L$ be the von Neumann subalgebra of $\M$ generated by $\Phi(\M)$ and $\A$. Then there is a projection $\Omega$ of norm $1$ from $\M $ onto $\L$.
Based on this we can construct a projection $\Gamma: \B(\H)\rightarrow \M$ by
$$\Gamma(T)=V^{\ast}\Omega(\Psi(VTV^{\ast}))V.$$
In order to estimate the norm $||\Gamma||_n$, we construct an isometry $W$ of $\H\oplus \H\oplus \cdots \oplus \H$ ($n$-times) onto $(E_{11}+E_{22}+\cdots +E_{nn})\H$ by defining a row matrix $$W=[E_{11}V\ E_{21}V\ \cdots,\ E_{n1}V]$$ in $M_{1\times n}(\B(\H))$. Then for any $X=[X_{ij}]_{i, j=1}^n\in M_n(\B(\H))$, we have
$$W\Gamma_n(X)W^{\ast}\in (E_{11}+E_{22}+\cdots +E_{nn})B(\H)(E_{11}+E_{22}+\cdots +E_{nn}).$$
Note $\Gamma_n(X)=[V^{\ast}\Omega(\Psi(VX_{ij}V^{\ast}))V]_{i, j=1}^n$ and then by the $\A$-modularity of $\Omega$ and $\Psi$, we get
\begin{align}W \Gamma_n(X) W^{\ast}&=\sum_{i, j=1}^nE_{i1}\Omega (\Psi(VX_{ij}V^{\ast}))E_{1j}
\nonumber\\
&= \Omega (\Psi(\sum_{i, j=1}^nE_{i1} VX_{ij}V^{\ast}E_{1j}))\nonumber\\
&=\Omega (\Psi(W[X_{ij}]_{i, j=1}^nW^{\ast})).\nonumber\end{align}
Hence \begin{align}
||\Gamma_n(X)||&=  ||W\Gamma_n([X_{ij}]_{i, j=1}^n) W^{\ast}||\nonumber\\
&= ||\Omega(\Psi(W[X_{ij}]_{i, j=1}^nW^{\ast}))||\nonumber\\
&
\leq ||\Psi||  ||[X_{ij}]_{i, j=1}^n||\nonumber\\
&
\leq ||\Pi||  ||X||.\nonumber
\end{align} This completes the proof of the lemma.
\end{proof}

\begin{proof} [Proof of Theorem 1]  From the lemmas above we see that for each $n\in \mathbb{N}$,
there exists a projection $\Gamma^n$ of $\B(\H)$ onto $\M$ such that $$||\Gamma^{n}||_{n}=||\Gamma_{n}^{n}||\leq ||\Pi||.$$
Since $\{\Gamma^{n}: n\in \mathbb{N}\}$ is a bounded sequence of projections of $\B(\H)$ onto ${\M}$, it has a subnet that converges pointwise ultraweakly  to a linear mapping $\Theta$ of $\B(\H)$ onto $ {\M}$. Hence $\Theta$ is a completely bounded projection from $\B(\H)$ onto $ \M $ with $||\Theta||_{cb}\leq ||\Pi||$, and
it follows from  \cite{CS1} or \cite{P3}  that $\M$ is an injective von Neumann algebra. This completes the proof.
\end{proof}

We will consider a type II$_1$ von Neumann algebra $\M$ but now also assume that $\M$ is a factor on a Hilbert space $\H$ which is a complemented subspace  of $\B(\H)$. We will present four results of the type "If $\M$ has a certain property, then $\M$ is injective". Only the first corollary is new, the second and the third appeared in \cite{C1}, \cite{P2}, \cite{P3} and the fourth in \cite{HP}. We have included new proofs here, because these results follow easily from our theorem.   We remind the readers of  the
four properties we want to look at.

\begin{definition} Let $\M$ be a factor of type $\mathrm{II}_1$.

(a) Let $\R$ denote the hyperfinite type $\mathrm{II}_1$ factor. Then $\M$ is said to be a McDuff factor if $\M\cong \M\overline{\otimes } \R$.

(b) $\M$ is said to be non-prime if $\M$ is isomorphic to the tensor product $\S\overline{\otimes }\T$  of two type $\mathrm{II}_1$ factors $\S$ and $\T$.
\end{definition}

In Kadison and Ringrose's book (Exercise 13.4.6 of \cite{KR}), the authors studied the fundamental group $\F(\M) \subseteq (\mathbb{R}^{+}, \times)$ of a type II$_1$ factor $\M$. We say that the fundamental group of $\M$ is nontrivial if $\F(\M)\neq \{1\}$.

As usual we will let $\L(\mathbb{F}_n)$ denote the  von Neumann algebra  generated by the left regular representation of the non-abelian free group $\mathbb{F}_n$ on $n$ generators ($2\leq n \leq \infty$).  Let $\L(\mathbb{F}_r)$ denote the interpolated free group factors for $r\in (1, +\infty]$.  We refer to the articles \cite{Dy} and \cite{Ra} for the constructions and properties of the interpolated free group factors.
We have obtained the following corollaries.

\begin{corollary} \label{QQQ} Let $ \M$ be a factor of type $\mathrm{II}_1$ on a Hilbert space $\H$.   If the fundamental group $\F(\M)\neq \{1\}$ and there is a bounded projection $\Pi$ of $\B(\H)$ onto $\M$, then $\M$ is  injective. \end{corollary}

\begin{proof}  Suppose   $\t$ is the normal faithful tracial state of $\M$.
Since $\F(\M)\neq \{1\}$, there exists a $t\in \F(\M)$, $0<t<1$. Thus there exists a nonzero projection $E$ in $\M$  such that $\t(E)=t$ and $\M\cong E\M E$. Let $\P: \M\rightarrow E\M E$ be a normal  $\ast$-isomorphism. Then the result follows from Theorem \ref{TTT}.
\end{proof}

\begin{corollary} \label{SSS} Let $ \M$ be a factor of type $\mathrm{II}_1$ on a Hilbert space $\H$. If $\M$ is McDuff and there is a bounded projection $\Pi$ of $\B(\H)$ onto $\M$, then $\M$ is  injective. \end{corollary}

\begin{proof}
If $\M$ is McDuff, then $\M\cong \M \overline{\otimes} \R$ which in turn implies that $\F(\M)=\mathbb{R}^{+}$ and  the result follows from Corollary \ref{QQQ}.
\end{proof}

\begin{corollary}   Let $ \M$ be a factor of type $\mathrm{II}_1$ on a Hilbert space $\H$. If $\M$ is non-prime and there is a bounded projection $\Pi$
from $\B(\H)$ onto $\M$, then $\M$ is an injective  factor.\end{corollary}

\begin{proof} Suppose  $\M=  \S \overline{\otimes} \T$ where $\S$ and $\T$ are two type II$_1$ factors. Then $\S$ contains a   copy $\R$   of the hyperfinite type II$_1$ factor with the same unit as $\S$. Hence $\M \supseteq \R \overline{\otimes } \T$ and since $\M$ is a type II$_1$ factor, the von Neumann algebra $\R \overline{\otimes } \T$ is complemented in $\M$ and then also complemented in $\B(\H)$.
On the other hand, $\R \overline{\otimes}\T$ is clearly a McDuff factor and then it is injective by Corollary \ref{SSS}. Since $\mathbb{C}I\overline{\otimes }\T$ is a type II$_1$ subfactor of the injective type II$_1$ factor $\R\overline{\otimes }\T$,  we see that  $\T$ is  injective and  by symmetry $\S$ is injective. Hence  $\M=\S\overline{\otimes} \T$ is also injective and the corollary follows. \end{proof}

The last corollary is due to Haagerup and Pisier and was presented in Corollary 4.7 of \cite{HP}.
We present   a proof which is based on Voiculescu's formula 
$$\L(\mathbb{F}_{k+1})\cong \L(\mathbb{F}_{n^2k+1})\otimes M_{n}(\mathbb{C})\hskip24pt (k, n\in\mathbb{N})$$
and its generalization to the interpolated free froup factors
$$\L(\mathbb{F}_{r})_t\cong \L(\mathbb{F}_{(r-1)t^{-2}+1}).\hskip24pt  (1< r\leq \infty, 0<t<\infty)$$
The latter formula was obtained independently by Dykema \cite{Dy} and Radulescu  \cite{Ra}.

\begin{corollary} \label{III} No   interpolated free group factor is isomorphic to a  complemented von Neumann algebra of type $\mathrm{II}_1$.
\end{corollary}



\begin{proof} Let $ \L(\mathbb{F}_r) $ be the interpolated free group factor for $r$ ($1<r\leq \infty$). Suppose $\M \subseteq  \B(\H)$ is a complemented type II$_1$ von Neumann algebra which is isomorphic to $\L(\mathbb{F}_r)$. Since $\L(\mathbb{F}_{r})_{\frac{1}{2}}\cong\L(\mathbb{F}_{4(r-1)+1})$ and $\L(\F_{r})\subseteq  \L(\mathbb{F}_{4(r-1)+1})$,  there is a injective $\ast$-isomorphism from $\M$ into $\M_{\frac{1}{2}}$. It follows from Theorem \ref{TTT} that $\M$ is injective which is a contradiction and the corollary follows.
\end{proof}

\begin{remark}  There exist factors of type $\mathrm{II}_1$ with trivial fundamental groups that are not complemented subspaces. \end{remark}

It follows from Popa's work \cite{Popa} that for each finite $n\geq 2$, there exists free, ergodic, measure-preserving actions $\sigma $  of $\mathbb{F}_n$ on $L^{\infty}([0, 1], \mu)$, where $\mu$ is the Lebesgue measure on $[0, 1]$, such that the crossed product $$\M=L^{\infty}([0, 1], \mu)\rtimes _{\sigma}\mathbb{F}_n$$ is a factor of type II$_1$ with trivial fundamental group, $\F(\M)=\{1\}$.

By construction, $\M$ contains a von Neumann subalgebra $\N$ which is isomorphic to $\L(\mathbb{F}_n)$.
Let $E: \M\rightarrow \N$ be the conditional expectation of $\M$ onto $\N$ (see \cite{KR} and \cite{To}). If there is a bounded projection $\Pi$ of $  \B(\H)$ onto $\M$, then $E\circ \Pi$ would be a bounded projection of   $B(\H)$ onto $\N$ which contradicts with Corollary \ref{III} (or Corollary 4.7 in    \cite{HP}). Hence there is no bounded projection from $\B(\H)$ onto $\M$.


\begin{thebibliography}{\sl 110}

\bibitem{BP} J. W. Bunce and W. L. Paschke, Quasi expectations and amenable von Neumann algebras, Proc. Amer. Math. Soc., 71(1978), 232-236.




\bibitem{CS1} E. Christensen, A. M. Sinclair, On von Neumann algebras which are complemented subspaces of $\B(\H)$, J. Funct. Anal., 122(1994), 91-102.

\bibitem{CS2}E. Christensen, A. M. Sinclair, Module mappings into von Neumann algebras and injectivity, Proc. London Math. Soc.,  71(1995), 618-640.

\bibitem{C1} E. Christensen, Finite von Neumann algebra factors with Property $\Gamma$, J. Funct. Anal., 186(2000), 366-380.

\bibitem{C1976} A. Connes, Classification of injective factors, Ann. of Math., 104(1976), 73-115.

\bibitem{C1978} A. Connes, On the cohomology of operator algebras,  J. Funct. Anal., 28(1978), 248-253.


\bibitem{Dix} J. Dixmier, Von Neumann algebras, North-Holland Publishing Company, 1981.


\bibitem{Dy} K. Dykema, Interpolated free group factors,     Pacific J. Math., 163(1994), 123-135.

\bibitem{Ell} G. A. Elliott, On approximately finite dimensional von Neumann algebras, II. Canad. Math. Bull., 21(1978), 415-418.

\bibitem{HP} U. Haagerup and G. Pisier, Bounded linear operators between C$^{\ast}$-algebras, Duke Math. Journ., 71(1993), 889-925.

\bibitem{KR} R. V. Kadison, J. R.  Ringrose,   Fundamentals of the theory of Operator Algebras,
Vol. I and II, Academic Press, Orlando, 1983 and 1986.

\bibitem{MVN36} F. J. Murray and J. Von Neumann. On rings of operators. Ann. of
Math.,   37(1936), 116-229.

\bibitem{MVN37} F. J. Murray and J. von Neumann. On rings of operators. II. Trans.
Amer. Math. Soc., 41(1937), 208-248.

\bibitem{MVN43} F. J. Murray and J. von Neumann. On rings of operators. IV. Ann. of
Math., 44(1943), 716-808.

\bibitem{Pa} V. Paulsen, Completely bounded maps and operator algebras, Cambridge Studies in Advanced Mathematics, 78.
Cambridge University Press, Cambridge, 2002.

\bibitem{P1} G. Pisier, Remarks on complemented subspaces of von Neumann algebras, Proc. Roy. Soc. Edinburgh Sect. A, 121(1992), 1-4.

\bibitem{P2} G. Pisier, Projections from a von Neumann algebra onto a subalgebra, Bull. Soc. Math. France, 123(1995), 139-153.

\bibitem{P3} G. Pisier, The operator Hilbert space $O(H)$ and complex interpolation over tensor norms, Mem. Amer. Math. Soc., 122(1996), no. 585.

\bibitem{P4}     G. Pisier, Remarks on the similarity degree of an operator algebra, Internat. J. Math.,
12(2001), 403-414.

\bibitem{Popa}     S. Popa, On the fundamental group of type II$_1$ factors, Proc. Nat. Acad. Sci., 101(2004), 723-726.

\bibitem{Ra}F. Radulescu, Random matrices, amalgamated free products and subfactors in free group factors of noninteger index, Inv. Math., 115(1994), 347-389.

\bibitem{T1979} M. Takesaki, Theory of operator algebras I, Springer Verlag, New York Inc. 1979.


\bibitem{To} J. Tomiyama, On the projection of norm one in W$^{\ast}$-algebras, Proc. Japan Acad., 33(1957), 608-612.



\end{thebibliography}
\end{document}